\documentclass{amsart}
\usepackage{amsfonts,amssymb,amscd,amsmath,enumerate,verbatim,calc}
\usepackage[all]{xy}

\newcommand{\CM}{Cohen-Macaulay}

\newcommand{\bP}{\mathbb{\partial}}
\newcommand{\cB}{\mathcal{B}}

\newcommand{\wrt}{with respect to}

\newcommand{\m}{\mathfrak{m} }

\newcommand{\PP}{\mathbb{P} }

\newcommand{\Ass}{\operatorname{Ass}}

\newcommand{\htt}{\operatorname{ht}}

\theoremstyle{plain}

\newtheorem{theorem}{Theorem}[section]
\newtheorem{corollary}[theorem]{Corollary}
\newtheorem{lemma}[theorem]{Lemma}
\newtheorem{proposition}[theorem]{Proposition}

\theoremstyle{definition}
\newtheorem{definition}[theorem]{Definition}
\newtheorem{s}[theorem]{}
\newtheorem{remark}[theorem]{Remark}
\newtheorem{example}[theorem]{Example}

\theoremstyle{remark}

\begin{document}

\title[De Rahm]{De Rahm Cohomology of Local Cohomology  modules II}

\author{Tony~J.~Puthenpurakal}
\author{Rakesh B. T. Reddy  }
\date{\today}
\address{Department of Mathematics, IIT Bombay, Powai, Mumbai 400 076}

\email{tputhen@math.iitb.ac.in} \email{rakesh@math.iitb.ac.in}

\begin{abstract}

Let $K$ be an algebraically closed  field of characteristic zero and
let $R = K[X_1,\ldots,X_n]$. Let $I$ be an ideal in $R$. Let
$A_n(K)$ be the $n^{th}$ Weyl algebra over $K$. By a result of
Lyubeznik, the local cohomology modules $H^i_I(R)$ are holonomic
$A_n(K)$-modules for each $i \geq 0$. In this paper we compute the
Euler characteristic of De-Rahm cohomology of $H^{\htt P}_P(R)$ for
certain classes of prime ideals $P$ in $R$.
\end{abstract}

\maketitle
\section*{Introduction}
Let $K$ be an algebraically closed  field of characteristic zero, $R
= K[X_1,\ldots,X_n]$ and let $I$ be an ideal in $R$. For $i \geq 0$
let $H^i_I(R)$ be the $i^{th}$-local cohomology module of $R$ with
respect to $I$. Let $A_n(K)  = K<X_1,\ldots,X_n, \partial_1, \ldots,
\partial_n>$  be the $n^{th}$ Weyl algebra over $K$. By a result due
to Lyubeznik, see \cite{L}, the local cohomology modules $H^i_I(R)$
are finitely generated $A_n(K)$-modules for each $i \geq 0$. In fact
they
 are \textit{holonomic} $A_n(K)$ modules. In \cite{BJ} holonomic $A_n(K)$ modules are denoted as $\cB_n(K)$, the  \textit{Bernstein} class of left $A_n(K)$ modules.

Let $N$ be a left $A_n(K)$ module. Now $\bP =
\partial_1,\ldots,\partial_n$ are pairwise commuting $K$-linear
maps. So we can consider the De Rahm complex $K(\bP;N)$. Notice that
the De Rahm cohomology modules  $H^*(\bP;N)$ are in general only
$K$-vector spaces. They are finite dimensional if $N$ is holonomic;
see \cite[Chapter 1, Theorem 6.1]{BJ}. In particular the Euler
characteristic $\chi^c(N) = \sum_{i = 0}^{n} \dim_K H^i(\bP; N)$ is
a finite number.
 In this paper we compute the
 Euler characteristic
  $\chi^c(H^g_P(R))$  for certain prime ideals of height $g$.

  We now describe the results of this paper.
Let   $P$
  be a height  $n-1$ prime ideal in $R$. Let $\mathcal{C} = V(P)$ in $\mathbb{A}^n(K)$ and let $\widehat{\mathcal{C}}$
  denotes
it's  projective closure in $\mathbb{P}^n(K)$. Set
$\deg(\mathcal{C}) = \deg(\widehat{\mathcal{C}})$.
   Let $z\in R$ sufficiently general
  linear form in $x_i's.$ Then in Theorem \ref{Main-Thrm-1} we prove that
\begin{center}
   $\chi^c(H^{n-1}_P(R)_z) -
  \chi^c(H^{n-1}_P(R))= (-1)^n\deg(V(P))$
  \end{center}
  If $S$ is a finite set then let $\sharp S$ denote the number of elements in $S$. Note that (with notation as above) $\widehat{\mathcal{C}} \setminus \mathcal{C}$ is a finite set of points which are called points of $C$ at infinity. Set
  $V_{\infty}(\mathcal{C}) = \sharp(\widehat{\mathcal{C}} \setminus \mathcal{C})$. In \ref{H_0-pts at inf}(a) we prove that
  \[
  \dim_K H^n(\bP; H^{n-1}_P(R)) \geq V_{\infty}(\mathcal{C}) - 1.
  \]

By Corollary \ref{proj-curves} we get that if $P$ is a graded prime
ideal of  height $n-2$ then
\[
\chi^{c}(H^{n-2}_P(R)) = (-1)^n.
\]
Prime ideals of height $n - 2$ correspond to curves in
$\PP^{n-1}(K)$. To say something about the Euler -characteristic for
higher degree varieties we need to make some more hypotheses. So let
$S = V(P)$ be a \CM \ surface, i.e., $\htt(P) = n-3$ and $Proj(R/P)$
is \CM. Note we have $H^i_P(R) = 0$ for $i \geq n-1$. It can be
shown that $H^{n-2}_P(R) = E_R(R/\m)^s$ for some $s \geq 0$; here
$\m = (X_1,\ldots,X_n)$ and $E_R(R/\m)$ is the injective hull of
$R/\m$; (see Lemma 4.5). In Theorem \ref{chi-HP^n-3=s-1} we prove
that
\[
\chi^{c}(H^{n-3}_P(R)) = (-1)^n(s-1).
\]
Thus $H^{n-3}_P(R)$ completely determines $H^{n-2}_P(R)$ if $V(P)$
is a \CM  \ surface. This is a completely unexpected result.

Next we consider the case when $V(P)$ is a $r$-dimensional
non-singular variety in $\PP^{n-1}(K)$ with $r \geq 2$. So $\htt(P)
= n -r -1$. It can be easily  shown that for $i \geq n-r$ the local
cohomology modules  $H^i_P(R) = E_R(R/\m)^{s_j}$ for some $s_j  \geq
0$; see Lemma 5.1. Our final result; Theorem 5.2;  is
\[
\chi^c(H^{n-r-1}_P(R) = (-1)^{n-r}\left( -1 + \sum_{j \geq
n-r}(-1)^{n-j}s_j \right).
\]
Thus the numbers $s_j$ cannot be arbitrary.

In section 1 we introduce notation and discuss a few preliminaries
that we need.
  We do not know of any  software to
   compute $H_i(\underline{\partial}:-)$. So in section 6 we compute
   $H_i(\underline{\partial}:H^1_{(f)}(R))$ for some irreducible
   polynomial $f\in K[x,y]$.
\section{Preliminaries}
In this section we discuss few preliminary results that we need.
\begin{remark}
Although all the results are stated for De-Rahm cohomology of a
$A_n(K)$-module $M$, we will actually work with De-Rahm homology.
Note that $H_i(\bP; M) = H^{n-i}(\bP; M)$ for any $A_n(K)$-module.
Also note that if
$$\chi(M) = \sum_{i = 0}^{n}(-1)^i\dim_K H_i(\bP; M),$$
 then $\chi(M) = (-1)^n\chi^c(M)$.
 Let $S = K[\partial_1,\ldots,\partial_n]$. Consider it as a subring of $A_n(K)$. Then note that $H_i(\bP; M)$ is the $i^{th}$ Koszul homology module of $M$ with respect to $\bP$.
\end{remark}

\begin{lemma}\label{H1-van}
Let $M$ be an $A_n(K)$-module. Set $M_0=ker(\partial_n:M\mapsto M)$
and $\bar{M}=M/\partial_nM.$ Set
$\underline{\partial}^{'}=\partial_1,\cdots,\partial_{n-1},$ and
$\underline{\partial}=\partial_1,\cdots,\partial_{n}.$ Then there
exist an exact sequence
\begin{center}
$\cdots \rightarrow H_{i-1}(\underline{\partial}^{'}; M_0)
\rightarrow H_i(\underline{\partial}; M) \rightarrow
H_{i}(\underline{\partial}^{'}; \bar{M})\rightarrow
H_{i-2}(\underline{\partial}^{'}; M_0)\rightarrow \cdots$
\end{center}
\end{lemma}
\begin{proof}
See Proposition $4.13$ in chapter $1$ of \cite{BJ} for cohomological
version.
\end{proof}
\begin{remark}\label{H1-van-rem}
Notice from the  exact sequence in \ref{H1-van}
$H_0(\underline{\partial}^{'}; \bar{M})= H_0(\underline{\partial};
M).$
\end{remark}
\begin{corollary}\label{H1-van-cor}
(With the hypotheses  as in Lemma \ref{H1-van}). If $M_0=0$ then
\[
H_i(\underline{\partial}^{'}; \bar{M}) \cong
H_i(\underline{\partial}; M) \quad \text{ for all } \quad i.\]
\end{corollary}
\begin{proof}
Clear from the above long exact sequence in homology.
\end{proof}
\s Let $\ell(-)=\dim_K(-)$. With notation as in Lemma \ref{H1-van}.
\begin{align*}
\text{Set} \quad &\chi(\underline{\partial}, M)=
\sum^n_{i=0}(-1)^i\ell(H_i(\underline{\partial}; M)),\\
&\chi(\underline{\partial}^{'}, \bar{M})=
\sum^{n-1}_{i=0}(-1)^i\ell(H_i(\underline{\partial}^{'}; \bar{M})),\\
\text{and}\quad &\chi(\underline{\partial^{'}}, M_0)=
\sum^{n-1}_{i=0}(-1)^i\ell(H_i(\underline{\partial}^{'}; M_0)).
\end{align*}
\begin{lemma}\label{chi-equi}
With the notation as above,
\begin{align*}
\chi(\underline{\partial}, M)=\chi(\underline{\partial}^{'},
\bar{M})-\chi(\underline{\partial}^{'}, M_0).
\end{align*}
\end{lemma}
\begin{proof}
Apply $\ell(-)$ to the exact sequence in Lemma \ref{H1-van}.
\end{proof}

 \s\label{Mayer-2} Let $I$ be an
ideal in $R$ and $f\in R.$ Then we have an exact sequence
\begin{center}
$\cdots \rightarrow (H^{i-1}_{I}(R))_f\rightarrow
H^i_{(I,f)}(R)\longrightarrow H^i_{I}(R)\longrightarrow
(H^i_{I}(R))_f\rightarrow \cdots$
\end{center}
of $A_n(K)-$modules. For proof see Lemma $1.6$  of \cite{TJ}.

The following is probably already known. We give a proof due to lack
of a reference.
\begin{lemma}\label{holon-M[z]}
Let $M$ be an $A_n(K)$-module and let $A=R[z]$. Set
$M[z]=\bigoplus_{i\geq 0}Mz^i$ and $M[z,z^{-1}]=\bigoplus_{i\in
\mathbb{Z}}Mz^i$. Then $M[z]$ and $M[z,z^{-1}]$ can be given
$A_{n+1}(K)$-module structure. Furthermore if $M$ is holonomic
$A_n(K)$-module then $M[z]$ and $M[z,z^{-1}]$ are holonomic
$A_{n+1}(K)$-modules.
\end{lemma}
\begin{proof}
Let $mz^i\in M[z]$ with $m\in M$. Define $z(mz^i)=mz^{i+1}$ for all
$i\geq 0$ and
\[
\partial_z(mz^i) =
\begin{cases} 0 & \text{if} \ i = 0 \\ imz^{i-1}& \text{if} \ i > 0.
\end{cases}
\]
One can easily check that with this action $M[z]$ becomes an
$A_{n+1}(K)$-module.  So  $M[z,z^{-1}]= M[z]_z$ is also an
$A_{n+1}(K)$-module.

 Let $M$ be a holonomic
$A_n(K)$-module. Then there exist a good filtration
$\mathcal{F}=(F_\upsilon)$ on $M$ which is compatible with Bernstein
filtration $\Gamma_\upsilon$ of $A_n(K) $ (see Proposition $2.7$
chapter $1$ of \cite{BJ}). Also there exist rational numbers
$a_0,\cdots, a_n$ such that
\begin{center}
$\ell(F_{\upsilon})=a_n\upsilon^n+\cdots+a_{1}\upsilon +a_0$ for all
$\upsilon \gg 0.$
\end{center}

Now define $G_{\upsilon}=F_{\upsilon}\oplus F_{\upsilon-1}z\oplus
F_{\upsilon-2}z^2\oplus \cdots\oplus F_oz^{\upsilon}.$ One  can
easily check that $G=(G_{\upsilon})$ is a filtration on $M[z]$ which
is compatible with the Bernstein filtration of $A_{n+1}(K).$

Note that
\begin{align*}
G_{\upsilon}/G_{\upsilon -1}=F_{\upsilon}/F_{\upsilon-1} \oplus
\cdots \oplus F_1/F_0 \oplus F_0.
\end{align*}
So
\begin{align*}
\sum_{\upsilon=0}^{\infty}\ell(G_{\upsilon}/G_{\upsilon
-1})z^{\upsilon}
&=1/(1-z)\sum_{\upsilon=0}^{\infty}\ell(F_{\upsilon}/F_{\upsilon
-1})z^{\upsilon}.\\
\Rightarrow \dim M[z]&= \dim M+1\\
&=n+1.
\end{align*}
So $M[z]$ is a holonomic $A_{n+1}(K)-$ module. Also by
 \cite[Theorem 1.5.5]{BJ},  $M[z,
 z^{-1}]$ is a holonomic $A_{n+1}(K)$-modules.
\end{proof}
\section{Curves}
The main result of this section is Theorem \ref{Main-Thrm-1}. The
following three results are well-known. We include here due to lack
of a reference.
\begin{lemma}\label{sngle-Ass-Lem}
Let $R=K[x_1, \cdots, x_n], $ and $P\subset R$ be a prime ideal of
height $g.$ Then Ass$_R(H_P^g(R))=\{P\}.$
\end{lemma}
\begin{proof}
First note that $(H_P^g(R))_P=H^g_{PR_P}(R_P)\cong
E_{R_P}(R_P/PR_P).$ As $PR_P\in$ Ass$_{R_P}E_{R_P}$ $(R_P/PR_P),$ we
have $PR_P\in$ Ass$_{R_P}(H_P^g(R))_P$ and so $P\in$
Ass$_R(H_P^g(R))$. Let $Q\in$ Ass$_R(H_P^g(R))$, so $Q=(0:r)$ for
some non-zero $r\in H_P^g(R)$. As $H^g_P(R)$ is $P-$ torsion we get
$P\subseteq Q$. If $P\neq Q$, choose $f\in Q$ such that $f\not \in
P.$ We have an exact sequence of the form
\begin{center}
 $H^g_{(P,f)}(R)\rightarrow H^g_P(R)\rightarrow (H^g_P(R))_f$.
\end{center}
  As ht$(P,f)=$ht$(P)+1,$ we get $H^g_{(P,f)}(R)=0$. Therefore  we have  an
exact sequence $0\rightarrow H^g_P(R)\stackrel{\psi}\longrightarrow
(H^g_P(R))_f$, where $\psi(m)=m/1$.  But $fr=0$ which implies
$r/1=0,$ which is a contradiction as $\psi$ is injective. Hence
$Q=P$ and Ass$_R(H_P^g(R))=\{P\}.$
\end{proof}

The following result is an easy consequence of the  Hartshorne-
Liechtenbaum vanishing theorem.
\begin{lemma}\label{HLV}
 Let $(R,\mathfrak m)$ be a regular local ring of dimension $n$. Let $I$
 be an ideal in $R$ which is not $\mathfrak m$-primary, then $H_I^n(R)=0.$
\end{lemma}
\begin{proof}
Let $\hat{R}$ be the $\mathfrak m$-adic completion of $R$. Then
$H_I^n(R)\bigotimes _R \hat{R}=H_{I\hat{R}}^n(\hat{R})$. As $I
\hat{R}$ is not $\hat{\mathfrak m}$ primary,  by Hartshorne-
Liechtenbaum vanishing theorem \cite[14.1]{SI} we get
$H_{I\hat{R}}^n(\hat{R})=0$. As $\hat{R}$ is faithfully flat $R$
module, we get  $H_I^n(R)=0.$
\end{proof}
As a consequence of \ref{HLV} we obtain
\begin{corollary}\label{HLV-1}
Let $R=K[x_1, \cdots, x_n]$, and $P\subset R$ be a prime ideal of
height $n-1$, $i.e.$ dim$\frac{R}{P}=1. $ Then $H_P^n(R)=0.$
\end{corollary}
\begin{proof}
Let $E:= H_P^n(R)$, then  $E$ is $P-$torsion. Let $\mathfrak m$ be
 a maximal ideal of $R$. If $P\nsubseteq \mathfrak m$, then
$E_{\mathfrak m}=0.$ If $P\subset \mathfrak m$, then $E_{\mathfrak
m}=H_{PR_{\mathfrak m}}^n(R_{\mathfrak m})=0$ by Lemma \ref{HLV}. So
$E=0.$
\end{proof}
\begin{lemma}\label{z-M-regular-Lem}
Let $M$ be a holonomic $A_n(K)$ module, and $z=x_n$ be $M-$regular.
Then $H_1(\partial_z; H_z^1(M) )=0,$ and so
\begin{center}
 $H_1(\partial_z; M)\cong
H_1(\partial_z; M_z), $
\end{center}
 and we have an exact
sequence
\begin{center}
 $0\rightarrow H_0(\partial_z;
M)\rightarrow H_0(\partial_z; M_z)\rightarrow H_0(\partial_z;
H_z^1(M))\rightarrow 0.$
\end{center}
\end{lemma}
\begin{proof}
 Notice that we have an exact sequence $0\rightarrow M
\rightarrow M_z \rightarrow H_z^1(M)\rightarrow 0 $. It suffices to
prove $H_1(\partial_z; H_z^1(M) )=0$.

  Let $\xi \in H_z^1(M) $ such that
$\partial_ z(\xi)=0.$ As $H_z^1(M)\cong M_z/M$, let $\xi
=\bar{m}/z^i$ for some $i\geq 0.$ Now
\begin{align*}
\partial_z(\xi)=0 \Rightarrow
\partial _z(m/z^i)&\in M\\
i.e. \quad \partial_z(m/z^i)&=n/1.\\
\text{ Therefore}\quad  \partial_z(m)/z^i  - i(m/z^{i+1})&=n/1.\\
\text{ Multiplying  by}\quad z \quad \text{we get}\quad
  m/z^i=n'/1+ & n_1/z^{i-1}.
\end{align*}
This shows $\xi =\bar{m}/z^i=\bar{n_1}/z^{i-1}$.
 Continuing in this way we get
 \begin{center}
$\xi=\bar{m}/z^i=\bar{n_1}/z^{i-1}=\bar{n_2}/z^{i-2}=\cdots
=\bar{n_i}/1=0$ in $M_z/M,$ where $n_i, m\in M.$
\end{center}
Hence $H_1(\partial_z; H_z^1(M) )=0.$
\end{proof}
\begin{definition}\label{def}
Let  $\mathcal{C}$ be an affine curve in $\mathbb{A}^n(K)$ and  let
$\hat{\mathcal{C}}$ be it's projective closure in $\mathbb{P}^n(K).$
Set degree$(\mathcal{C})$=degree$(\hat{\mathcal{C}}).$
\end{definition}
The main result of this section is:
\begin{theorem}\label{Main-Thrm-1}
Let  $R=K[x_1, \cdots, x_n]$, $P\subset R$ be a  prime ideal of
height $n-1$,  and let $z$ be a sufficiently general linear form in
$x_i^,s$. Then $\chi((H_P^{n-1}(R))_z)-
\chi(H_P^{n-1}(R))=e=$degree$(\mathcal{C})$, where
$\mathcal{C}=V(P).$
\end{theorem}

\begin{proof}
 Let $\hat{\mathcal{C}}$ be the  projective closure of  $\mathcal{C}$. We know that if $H$
 is a
sufficiently general hyperplane then $H\cap \hat{\mathcal{C}}$ has
degree$(\hat{\mathcal{C}})$ distinct points. We may also assume that
none of these points are in hyperplane at infinity. So let $z\notin
P$  be a sufficiently general linear form in $x_i^,s$, then $V(P,z)$
has degree($\mathcal{C}$) distinct points. By Lemma
\ref{sngle-Ass-Lem}, Ass$_R(H^{n-1}_P(R)) = \{P\}$. Since $z\not \in
P$,
  it is $H^{n-1}_P(R)$-regular.

 Let $I:=(P, z)$, notice
that $I$ is zero dimensional $i.e.$ $\dim (R/I)=0.$ We have an exact
sequence of the form
\begin{center}
 $\cdots \rightarrow H_{(P,z)}^{n-1}(R)\rightarrow
H_P^{n-1}(R)\rightarrow (H_P^{n-1}(R))_z\rightarrow
H_{(P,z)}^n(R)\rightarrow H_P^n(R)\rightarrow \cdots$
\end{center}
Notice that  $H_{(P,z)}^{n-1}(R)=0$, and by corollary \ref{HLV-1}
$H_P^n(R)=0$. So we get a short exact sequence
\begin{equation}\label{H_P-H_(P,z)-ses}
0\rightarrow H_P^{n-1}(R)\rightarrow (H_P^{n-1}(R))_z\rightarrow
H_{(P,z)}^n(R)\rightarrow 0.
\end{equation}
Therefore $H^n_{(P,z)}(R)\cong (H_P^{n-1}(R))_z/H_P^{n-1}(R).$  So
$H^n_{(P,z)}(R)\cong H^1_{(z)}(H^{n-1}_P(R)).$

 Let
$\hat{x}_{1},\cdots ,\hat{x}_{n-1},
 z=\hat{x}_n$ be a linear change of variables, then by \cite[1.4]{TJ}
\begin{center}
  $H_i(\frac{\partial}{\partial x_1},\cdots,$ $\frac{\partial}{\partial x_n};-)\cong H_i(\frac{\partial}{\partial
  \hat{x}_1},\cdots,\frac{\partial}{\partial
  \hat{x}_n};-)$.
\end{center}
By taking $H_*(\partial_z;-)$ to (\ref{H_P-H_(P,z)-ses}) and using
Lemma \ref{z-M-regular-Lem} we get
\begin{equation}\label{H_P-H_Pz-iso}
 H_1(\partial_z; H_P^{n-1}(R))\cong
 H_1(\partial_z; (H_P^{n-1}(R))_z),\quad \text{and an exact sequence}
\end{equation}
\begin{equation}\label{H_P-H_(P,z)-ses-2}
0 \rightarrow H_0( \partial_z; H_P^{n-1}(R))\rightarrow
H_0(\partial_z; (H_P^{n-1}(R))_z)\rightarrow H_0(\partial_z;
H_{(P,z)}^n(R))\rightarrow 0 .
\end{equation}

    Set $\underline{\partial} =\partial_1,\cdots,\partial_{n}$, and
$\underline{\acute{\partial}}=\partial_1,\cdots,\partial_{n-1}$.
Note that  $V(P,z)$ is a finite set. Let $\sharp V(P, z)=$
degree$(\hat{\mathcal{C}})= e$. By Theorem $1$ of \cite{TJ}, we get
\begin{equation}\label{Thrm1-paper1}
H_i(\underline{\partial};H^n_{(P,z)}(R)) =
\begin{cases} 0 & \text{for} \ i > 0 \\ K^e & \text{for} \ i = 0
\end{cases}
\end{equation}
 Also note that, as $H_1(\partial_z; H^n_{(P,z)}(R))=0$, by Lemma\ref{H1-van} we get an isomorphism
 \begin{equation}\label{HiH_0-H_0}
H_i(\underline{\acute{\partial}};H_0(\partial_z;
H^n_{(P,z)}(R)))\cong H_i(\underline{\partial};H^n_{(P,z)}(R))\quad
\text{for all} \quad i\geq 0.
\end{equation}
\begin{align*}
\text{Now set} \quad M_1: & =H_0(\partial_z ; H_P^{n-1}(R)),\\
 M_2:&=H_0(\partial_z;(H_P^{n-1}(R))_z),\\
 \text{ and}\quad   M_3: & =H_0(\partial_z;H_{(P,z)}^n (R)).
\end{align*}
  Note that $M_1, M_2, M_3$ are holonomic
$A_{n-1}(K)$ modules.  By taking  Koszul homology with respect to
$\underline{\acute{\partial}}$ to (\ref{H_P-H_(P,z)-ses-2}) we
obtain
\[-\sum_{i=0}^{n-1}(-1)^il(H_i(\underline{\acute{\partial}}; M_1)) +
\sum_{i=0}^{n-1}(-1)^il(H_i(\underline{\acute{\partial}};
M_2))=\sum_{i=0}^{n-1}(-1)^il(H_i(\underline{\acute{\partial}};
M_3)).\]
\[
 \text{By (\ref{Thrm1-paper1}) and (\ref{HiH_0-H_0}) we have}\quad
\sum_{i=0}^{n-1}(-1)^il(H_i(\underline{\acute{\partial}} ;
M_3))=l(H_0(\underline{\partial};H^n_{(P,z)}(R)))=e. \] So we get
\begin{equation}\label{length=e}
\sum_{i=0}^{n-1}(-1)^il(H_i(\underline{\acute{\partial}};
M_2))-\sum_{i=0}^{n-1}(-1)^il(H_i(\underline{\acute{\partial}};
M_1))=e.
\end{equation}
By (\ref{H_P-H_Pz-iso}) we have
\begin{equation}\label{after-e}
\quad H_i(\underline{\acute{\partial}}; H_1(\partial_z ;
H_P^{n-1}(R))) \cong H_i(\underline{\acute{\partial}};
H_1(\partial_z; (H_P^{n-1}(R))_z),\text{ for}\quad 0\leq i \leq n-1.
\end{equation}
 Hence the
theorem follows from Lemma \ref{chi-equi}, (\ref{length=e}) and
(\ref{after-e}).
\end{proof}

\section{Homogenization and De-homogenization}
In this section we consider the projective closure of an affine
curve in $\PP^n(K)$. In this section our main result is Theorem
\ref{H_0-pts at inf}. An easy but significant corollary of this
result is Corollary \ref{proj-curves}

 \textbf{Homogenization and De-homogenization}: Let us recall the notion of homogenization and
de-homogenization. Let $R=K[x_1,\cdots x_n,z], A=K[x_1,\cdots x_n]$
and $S=R_z=A[z,z^{-1}]$. Then we know that $A\hookrightarrow S$ is
$A$ flat. For an ideal $J$  in $A$, set $I=JS\cap R.$ Then $IS=JS$
and $IS\cap A=JS\cap A=J.$  The ideal $I$ is said to be
homogenization of $J$ with respect to $z$. Also, if $I$ is an  ideal
in $R$ then $J=IS\cap A$ is said to be de-homogenization of $I$ with
respect to $z.$  Now $(H_I^i(R))_z
=H_{I_z}^i(R_z)=H_{JS}^i(S)=H_{JS}^i(A[z,z^{-1}])=H_J^i(A)\bigotimes_AA[z,z^{-1}]=(H_J^i(A))[z,z^{-1}].$
For details (see $1.5.26$ of \cite{BH})

The following result is useful.
\begin{lemma}\label{M[z,z^-1]-M-iso}
Let $R=K[x_1,\cdots, x_n,z]$,  $A=K[x_1,\cdots, x_n],$ and  $M$ be a
holonomic $A_n(K)$ module. Then $H_i(\partial_z; M[z,z^{-1}])\cong
M$, for $i=0,1,$ as $A_n(K)-$modules.
\end{lemma}
\begin{proof}
Let $u\in H_1(\partial_z; M[z,z^{-1}]) $. As
$M[z,z^{-1}]=\bigoplus_{i\in \mathbb{Z}}Mz^i$, write
\begin{align*}
u &=\sum_{j=1}^m\frac{a_j}{z^j}+b_0+\sum_{i=1}^n b_iz^i.\\
\text{ Then}\quad
\partial_z(u)&=\sum_{j=1}^m\frac{-ja_j}{z^{j+1}}+0+\sum_{i=1}^{n}ib_iz^{i-1}
.
\end{align*}
 So $\partial_z(u)=0\Leftrightarrow a_j=0$ and
$b_i=0 $ for all $j=1,\cdots, m$, $i=1,\cdots, n.$  So $u=b_0\in M.$
Clearly  $M\subset H_1(\partial_z; M[z,z^{-1}]).$ Therefore
$H_1(\partial_z; M[z,z^{-1}])=M.$

 Now, let $u=b_0+\sum_{j=1}^n
b_jz^j$ and take $v=b_0z+\sum_{j=1}^nb_jz^{j+1}/(j+1)$, then
$\partial_z(v)=u.$ For $u=\sum_{j=2}^ma_j/z^j,$
\begin{align*}
 \text{take}\quad
v=\sum_{j=1}^{m-1} \frac{b_j}{z^j}, \text{ then} \quad
\partial_z(v)=\sum_{j=1}^{m-1}\frac{-jb_j}{z^{j+1}}.
\end{align*}
 Set $r=j+1$ and
$b_{r-1}=a_r/(1-r), $ we get $\partial_z(v)=u.$ So for $\theta \in
M[z,z^{-1}],$ we have, $\theta =a/z+\partial_z( \alpha),$ for some
$a\in M$ and $\alpha \in M[z,z^{-1}]$, so $\bar{\theta}=a/z$, for
some $a\in M.$ Thus $H_0(\partial_z; M[z,z^{-1}])\cong M/z\cong M$
as $A_n(K)$ modules.
\end{proof}
 \textbf{Set-up.}  Let
$A=K[x_1,\cdots,x_n]$, $R=K[x_1,\cdots,x_n,z]=A[z].$ Let $P\subset
A$ be a prime ideal of height ${n-1}$. Let  $ P^*=<f^* \mid f\in P>$
be the homogenization of $P$. We know that ht$(P^*)$=ht$(P) = n-1$.
Let $\mathcal{C}=V(P)$. Then $\hat{\mathcal{C}}=V(P^*).$ Furthermore
$\hat{\mathcal{C}}\setminus \mathcal{C}$ is a finite set of points
which are called points of $\mathcal{C}$ at  $\infty$. Set
$V_{\infty}(\mathcal{C})=\sharp \hat{\mathcal{C}}\setminus
\mathcal{C}.$\\
We need the following result in the proof of Theorem \ref{H_0-pts at
inf}.
\begin{lemma}\label{lem-H_0-pts at inf}
(With the above set-up) if ht$(P)=n-1$ $ i.e.$ $\dim(A/P)=1.$  Then
$H_{P^*}^n(R)=0.$
\end{lemma}
\begin{proof}

 Let $\mathfrak{m}=(x_1,\cdots, x_n, z).$  As $R/P^*$ is
 a domain, so  is $\hat{R}_{\mathfrak{m}}/P^*\hat{R}_{\mathfrak{m}}$
 is a domain of dimension $=2.$

 By \cite[14.7 and 15.5]{SI}
    $H^n_{P^*\hat{R}_{\mathfrak{m}}}(\hat{R}_{\mathfrak{m}})=0.$
    Thus
 $H^n_{P^*R_{\mathfrak{m}}}(R_{\mathfrak{m}})\otimes_{R_{\mathfrak{m}}}\hat{R}_{\mathfrak{m}}=0.$ As $\hat{R}_{\mathfrak{m}}$ is
 faithfully flat $R_{\mathfrak{m}}$- algebra, so
 $H^n_{P^*R_{\mathfrak{m}}}(R_{\mathfrak{m}})=0.$ So
 $H^n_{P^*}(R)\otimes_RR_{\mathfrak{m}}=0.$ As $H^n_{P^*}(R)$ is graded $R$ module
 and $-\otimes_RR_{\mathfrak{m}}$ is faithfully exact functor on graded $R$
 modules. Hence $H^n_{P^*}(R)=0.$
\end{proof}

\begin{theorem}\label{H_0-pts at inf}
(With the above set-up) if ht$(P)=n-1$. Then\\
(a) $\dim_K(H_0(\underline{\partial};H_P^{n-1}(A))\geqslant$
$V_{\infty}(\mathcal{C})-1$.\\
(b) $\chi(H_{P^*}^{n-1}$ $(R))=1.$
\end{theorem}
\begin{proof}
$(a)$ By Lemma \ref{lem-H_0-pts at inf} $H_{P^*}^n(R)=0$, and
$(H_{P^*}^{n-1}(R))_z=H_P^{n-1}(A)[z,z^{-1}]$,  so we have an exact
sequence
\begin{center}
$0\rightarrow H_{P^*}^{n-1}(R)\rightarrow
(H_{P^*}^{n-1}(R))_z\rightarrow H_{(P^*,z)}^n(R)\rightarrow 0. $
\end{center}

Thus $H_{(P^*,z)}^n(R)\cong H^1_z(H_{P^*}^{n-1}(R))).$ By taking
Koszul homology with respect to $\partial_z$ and using Lemma
\ref{z-M-regular-Lem} we get $H_1(\partial_z; H_{(P^*, z)}^n(R))=0$.
So

\begin{equation}\label{a-2}
H_1(\partial_z;H_{P^*}^{n-1}(R))\equiv
H_1(\partial_z;(H_{P^*}^{n-1}(R))_z),
\end{equation}
and we have an exact sequence
\begin{equation}\label{a-ses}
0\rightarrow H_0(\partial_z;H_{P^*}^{n-1}(R))\rightarrow
H_0(\partial_z;H^{n-1}_P(A)[z,z^{-1}])\rightarrow
H_0(\partial_z;H_{(P^*, z)}^{n}(R))\rightarrow 0.
\end{equation}
Also by  Lemma \ref{M[z,z^-1]-M-iso} we have $
H_0(\partial_z;H^{n-1}_P(A)[z,z^{-1}])\cong H^{n-1}_P(A)$ as
$A_n(K)$ modules. With this identification, and taking Koszul
homology to (\ref{a-ses}) with respect to $\partial
=\partial_{x_1},\cdots
\partial_{x_n}$, we get an exact sequence
\begin{center}
$H_0(\partial ; H_P^{n-1}(A))\rightarrow H_0(\partial;H_0(\partial_z
;H_{(P^*, z)}^{n}(R)))\rightarrow 0.$
\end{center}
So that $\dim_KH_0(\partial ; H_P^{n-1}(A))\geqslant$
$\dim_KH_0(\partial;H_0(\partial_z ;H_{(P^*, z)}^{n}(R)))$. By
Remark \ref{H1-van-rem} we have $H_0(\partial ; H_0(\partial_z
;H_{(P^*, z)}^{n}(R)))=H_0(\partial,
\partial_z ;H_{(P^*, z)}^{n}(R)).$ As ht$(P^*,z)$ $=n$, by
\cite[Theorem 2]{TJ} \   we get $\dim_KH_0(\partial, \partial_z
;H_{(P^*, z)}^{n}(R))=\sharp V(P^*,z)-1. $  \  \  Notice that \ \
$\sharp$ $ V(P^*,z)=\sharp$ points of $P$ at $\infty.$

 $(b)$ By
Lemma \ref{M[z,z^-1]-M-iso} we have $H_1(\partial_z
;(H_{P^*}^{n-1}(R))_z)=H^{n-1}_P(A)$. So  by (\ref{a-2}) we get
\[
\sum_{i=0}^n(-1)^i\ell(H_i(\partial ;H_1(\partial_z
;H_{P^*}^{n-1}(R))))=\sum_{i=0}^n(-1)^i\ell(H_i(\partial
;H_P^{n-1}(A))).
\]
By (\ref{a-ses}) we get
\begin{align*}
\sum_{i=0}^n(-1)^i\ell(H_i(\partial;H_P^{n-1}(A)))&=\sum_{i=0}^n(-1)^i\ell(H_i(\partial;
H_0(\partial_z;H_{(P^*,z)}^n(R))))\\
&
+\sum_{i=0}^n(-1)^i\ell(H_i(\partial;H_0(\partial_z;H_{P^*}^{n-1}(R)))).
\end{align*}
Now
  \begin{align*}
 \chi(H_{P^*}^{n-1}(R))&=\sum^{n+1}_{i=0}(-1)^i\ell(H_i(\partial,\partial_z;H_{P^*}^{n-1}(R)))\\
&=-\sum^n_{i=0}(-1)^i\ell(H_i(\partial;H_1(\partial_z;H_{P^*}^{n-1}(R))))\\
 &
+\sum^n_{i=0}(-1)^i\ell(H_i(\partial;H_0(\partial_z;H_{P^*}^{n-1}(R))))\\
&=-\sum^n_{i=0}(-1)^i
\ell(H_i(\partial;H_0(\partial_z;H_{(P^*,z)}^n(R))))\\
& =-\chi(H_{(P^*,z)}^n(R))=1\quad \text{(by Theorem 2 of
\cite{TJ})}.
\end{align*}
\end{proof}

An easy consequence of Theorem \ref{H_0-pts at inf} is the
following:
\begin{corollary}\label{proj-curves}
Let $P$ be a height $n-1$ graded prime ideal in $R= K[x_0,
x_1,\ldots,x_n]$. Then
$$ \chi(H^{n-1}_P(R)) = 1.$$
\end{corollary}
\begin{proof}
Let $x$ be a sufficiently general homogeneous linear form in the
$x_i$'s. By linear change of variables we may assume $x = x_0$. We
de-homogenize w.r.t. $x_0$. Set $A = R/(x_0-1)$ and let $Q$ be the
image of $P$ in $A$. Then note that after homogenizing w.r.t. $x_0$
we get $Q^* = P$. So we get the result from Theorem \ref{H_0-pts at
inf}(b).
\end{proof}

\section{Cohen-Macaulay Surfaces}
The main result of this section is Theorem \ref{chi-HP^n-3=s-1}. The
following result is well known. Let $Spec^0(R)= Spec(R)\backslash
\{\mathfrak{m}\}$ denote the punctured spectrum of $R.$

\begin{proposition}\label{Proj-Spec^0}
Let $R=\oplus_{n\geq 0}R_n$ be a standard graded ring.
$\mathfrak{m}=\oplus_{n\geq 1}R_n.$ Then the following are
equivalent\\
$(1)$ Proj($R$) is connected.\\
$(2)$ Spec$^0(R_\mathfrak{m})$ is connected.\\
$(3)$ Spec$^0(\hat{R}_\mathfrak{m})$ is connected.
\end{proposition}

The following result is an easy consequence of
Hartshorne-Leichtenbaum vanishing theorem \cite[Theorem 14.1]{SI}.

\begin{lemma}\label{H_P^n-1-E}
Let $I$ be an unmixed  graded ideal of $R=K[x_1,\cdots,x_n]$. Assume
that ht$(I)=n-2$. Let $\mathfrak{m}=(x_1,\cdots,x_n).$ Then
$H^{n-1}_I(R)=E_R(\frac{R}{\mathfrak{m}})^s$ for some $s\geq 0.$
\end{lemma}
\begin{proof}
By Hartshorne-Leichtenbaum vanishing theorem it follows that
$H^{n-1}_I(R)$ is supported only at maximal ideals of $R$. As
$H^{n-1}_I(R)$ is graded, it follows that
Ass$_RH^{n-1}_I(R)=\{\mathfrak{m}\}$. Thus $H^{n-1}_I(R)$ is
supported only at $\mathfrak{m}$. So $H^{n-1}_I(R)\cong
E_R(\frac{R}{\mathfrak{m}})^s$ for some $s\geq 0.$
\end{proof}
The following result was one of the motivation to prove Theorem
\ref{chi-HP^n-3=s-1}. It is also needed in it's proof.
\begin{theorem}\label{HI^n-1-E^s}
Let $I$ be a  graded  ideal of $R=K[x_1,\cdots,x_n]$ with
ht$(I)=n-2$.  Let $\mathfrak{m}=(x_1,\cdots, x_n).$ If
$H_I^{n-1}(R)=E_R(\frac{R}{\mathfrak{m}})^s$ for some $s \geq 0$
then $\chi(H_I^{n-2}(R))=1+s$. In particular if $I$ is unmixed then
$\chi(H_I^{n-2}(R))=1\Leftrightarrow cd(I)=n-2.$
\end{theorem}
\begin{proof}
  Let $z$ be a homogeneous linear form  and $ R/I$ regular. Set $A=R/(z-1)$.

Let $J= I_*$ be the de-homogenization of $I.$ We have an exact
sequence
\begin{center}
$0\rightarrow H_I^{n-2}(R)\rightarrow
H_{I_z}^{n-2}(R_z)(=H_J^{n-2}(A)[z,z^{-1}])\rightarrow
H_{(I,z)}^{n-1}(R)\rightarrow H_I^{n-1}(R)\rightarrow
(H_I^{n-1}(R))_z\rightarrow H_{(I,z)}^n(R)\rightarrow \cdots$
\end{center}
Since  $z$ is  homogeneous, and $E_R(R/\mathfrak{m})^s$ is
$\mathfrak{m}$ torsion, so $((E_R(R/\mathfrak{m}))^s)_z=0.$  Thus
$(H_I^{n-1}(R))_z =0.$
 So we get an exact sequence of the form
\begin{center}
$0\rightarrow H_I^{n-2}(R)\rightarrow
H_J^{n-2}(A)[z,z^{-1}]\stackrel{\alpha}\longrightarrow
H_{(I,z)}^{n-1}(R)\rightarrow E_R(R/\mathfrak{m})^s\rightarrow 0.$
\end{center}
By \ref{chiH_P^n-2=0}, $\chi(H^{n-2}_J(A)[z, z^{-1}])=0 $. Also by
\cite[Theorem 2]{TJ} we have $\chi(H_{(I,z)}^{n-1}(R)) = -1$. As
$\chi(-)$ is additive \wrt \  exact sequences we get
\begin{align*}
\chi(H_I^{n-2}(R))&=\chi(H_J^{n-2}(A)[z,z^{-1}])-\chi(H_{(I,z)}^{n-1}(R))+\chi(E_R(R/\mathfrak{m})^s)\\
&=0+1+s.
\end{align*}
   Hence the result.
\end{proof}

\begin{lemma}\label{chiH_P^n-2=0}
Let $N$ be a holonomic $A_{n-1}(K)$-module. $M=N[z,z^{-1}].$ Then
$\chi(M)=0$.
\end{lemma}
\begin{proof}
 Set $M_0 =
H_1(\partial_z;M)$, and $\bar{M}=H_0(\partial_z;M).$ By Lemma
\ref{chi-equi} we have
 \[\chi(M) = \chi(\bar{M})-\chi(M_0).\]
As $M=N[z,z^{-1}]$. By Lemma \ref{M[z,z^-1]-M-iso} $\bar{M}= M_0=N$.
So $\chi(M)=0.$

\end{proof}

\begin{corollary}\label{H_P^n-1-E-cor}
 Let $I$ be an unmixed graded
ideal of $R=K[x_1,\cdots,x_n]$. Let ht$(I)=n-2$. Then
$\chi(H_I^{n-2}(R))=1$ if Proj$(R/I)$ is connected or if $R/I$ is
Cohen-Macaulay.
\end{corollary}
\begin{proof}
As ht$(I)= n-2$ so $\dim (R/I)=2.$ We get $\dim
(R/I)_{\mathfrak{m}}=2.$ Therefore
$\dim(\hat{R}_{\mathfrak{m}}/I\hat{R_{\mathfrak{m}}})=2.$ As
Proj$(R/I)$ is connected, by Proposition \ref{Proj-Spec^0} we get
Spec$^0(\hat{R}_{\mathfrak{m}}/I\hat{R_{\mathfrak{m}}})$ is
connected. Therefore by  \cite[14.7]{SI} we get
$H^{n-1}_{I\hat{R}_{\mathfrak{m}}}(\hat{R}_{\mathfrak{m}})=0.$
Similarly if $R/I$ is Cohen-Macaulay of $\dim = 2$ then
$\hat{R}/\hat{I}$ is Cohen-Macaulay of $\dim =2$. So
$Spec^0(\hat{R}/\hat{I})$ is connected by proposition
\ref{Proj-Spec^0}.  As $\hat{R}_{\mathfrak{m}}$ is faithfully flat
$R_{\mathfrak{m}}$-algebra so
$H^{n-1}_{IR_{\mathfrak{m}}}(R_{\mathfrak{m}})=0.$ Also note that
$H^{n-1}_{I}(R)$ is graded $R$ module
 and $-\otimes_RR_{\mathfrak{m}}$ is faithfully exact functor on graded $R$
 modules. Therefore
$H^{n-1}_{I}(R)=0.$ Hence by Theorem \ref{HI^n-1-E^s},
$\chi(H^{n-2}_I(R))=1.$
\end{proof}

\begin{lemma}\label{htI=n-3}
Let $R=K[x_1,\cdots,x_n]$, $\mathfrak{m}=(x_1,\cdots,x_n).$ Let $I$
be an unmixed, graded, height $n-3$ ideal in $R.$ Suppose
Proj$(R/I)$ is Cohen-Macaulay. Then
$H^{n-2}_I(R)=E_R(R/\mathfrak{m})^s$ for some $s\geq 0.$
\end{lemma}

\begin{proof}
It suffices to show that Ass$_RH^{n-2}_I(R)\subset
\{\mathfrak{m}\}.$ Let $P\in $ Ass$_RH^{n-2}_I(R).$ As
$H^{n-2}_I(R)$ is graded, so $P$ graded prime. Note that $I\subset
P.$ Thus ht$(P)\geq n-3.$

 \textit{Case}$(1):$ ht$(P)=n-3.$ Then
$\dim R_P = n-3$. Therefore $(H^{n-2}_I(R))_P  =
H^{n-2}_{IR_P}(R_P)=0$ by Grothendieck vanishing theorem. Therefore
$P\not \in$Ass$_R(H^{n-2}_I(R))$.

\textit{Case}$(2):$ Let ht$(P)=n-2.$ As $I$ is unmixed and
ht$(I)=n-3$ we get $IR_P$ is not primary to $PR_P.$ This implies
$I\hat{R}_P$ is not primary to $P\hat{R}_P.$  By Hartshorne-
Liechtenbaum vanishing theorem \cite[14.1]{SI} we get
$H^{n-2}_{I\hat{R}_P}(\hat{R}_P)=0.$ Also note that
$H^{n-2}_{IR_P}(R_P)\otimes_{R_P}\hat{R}_P=H^{n-2}_{I\hat{R}_P}(\hat{R}_P),$
and $\hat{R}_P$ is faithfully flat $R_P$ algebra. Therefore
$(H^{n-2}_I(R))_P = H^{n-2}_{IR_P}(R_P)=0.$ Thus $P\notin
\Ass_R(H^{n-2}_I(R)).$

 \textit{Case}$(3):$ Suppose ht$(P)=n-1.$ As Proj$(R/I)$ is
Cohen-Macaulay, and $P\in$ Proj$(R/I),$ we get $(R/I)_P=R_P/IR_P$ is
Cohen-Macaulay.

 By \cite[App. Th\'{e}or\'{e}me 1, Corollaire]{BourIX}, there exist a complete regular local ring $(S, \mathfrak{n})$ such that\\
$(1)$ $R_P\hookrightarrow S$ is flat and $\dim S=\dim R_P$.\\
$(2)$ $PR_PS=\mathfrak{n}.$\\
$(3)$ $S/\mathfrak{n} =$ algebraic closure of $R_P/PR_P.$\\
As $R_P/IR_P$ is Cohen-Macaulay, we get $R_P/IR_P\otimes_{R_P}S$
$\quad$is Cohen-Macaulay. Thus $S/IR_PS$ is Cohen-Macaulay of
dimension $2.$ Therefore Spec$^0(S/IR_PS)$ is connected. Note that
$\dim S= \dim R_P=n-1.$ We get $H^{n-2}_{IR_PS}(S)=0.$ So
$H^{n-2}_{IR_P}(R_P)\otimes_{R_P}S=H^{n-2}_{IR_PS}(S)=0.$ Also $S$
is faithfully flat extension of $R_P.$ Therefore
$(H^{n-2}_I(R))_P=H^{n-2}_{IR_P}(R_P)=0.$ Hence $P\not
\in$Ass$_R(H^{n-2}_I(R)).$

\end{proof}

\begin{theorem}\label{chi-HP^n-3=s-1}
Let $R=K[x_1,\cdots,x_n]$ and $P\subset R$ be a prime ideal. Suppose
$V(P)$ is a Cohen-Macaulay surface i.e. ht$P=n-3,$ and Proj$(R/P)$
is Cohen-Macaulay. Then $\chi(H^{n-3}_P(R))= s-1,$ where
$H^{n-2}_P(R)=E(R/\mathfrak{m})^s.$
\end{theorem}

\begin{proof}
Let $z$ be a homogeneous linear form which is $R/P$ regular.  Note
$E(R/\m)_z = 0$. Let $Q:=$ de-homogenization of $P$ with respect to
$z$. Set $A=R/(z-1).$ Now consider the exact sequence
\begin{center}
$0\rightarrow H^{n-3}_P(R)\rightarrow (H^{n-3}_P(R))_z\rightarrow
H^{n-2}_{(P+(z))}(R)\rightarrow H^{n-2}_P(R)\rightarrow
(H^{n-2}_P(R))_z=0;$
\end{center}
We also get $H^{n-1}_{(P+(z))}(R)$ is a sub-module of
$H^{n-1}_P(R)$. By an argument similar to Lemma 3.4 we get that
$H^{n-1}_P(R) = 0$. So $H^{n-1}_{(P+(z))}(R) =0$. By
\ref{H_P^n-1-E-cor} we get $\chi(H^{n-2}_{(P+(z))}(R))=1.$ Note that
\begin{align*}
 H^{n-3}_{P_{z}}(R_z) & = H^{n-3}_{QS}(S),\quad \text{where}\quad
S=A[z, z^{-1}].\\
& =H^{n-3}_{Q}(A)[z, z^{-1}].
\end{align*}
By \ref{chiH_P^n-2=0} we get $\chi(H^{n-3}_{Q}(A)[z, z^{-1}])=0.$ As
$\chi(E(R/\mathfrak{m})^s)=s,$ by taking $\chi(-)$ in the above
exact sequence we get
\begin{align*}
\chi(H^{n-3}_P(R))&=s-\chi(H^{n-2}_{P+(z)}(R))\\
&=s-1.
\end{align*}
\end{proof}

The proof of the following result is similar to the proof of Theorem
\ref{chi-HP^n-3=s-1}.
\begin{theorem}\label{chi-HI^n-3=s-1}
Let $R=K[x_1,\cdots,x_n]$,$\mathfrak{m}=(x_1,\cdots,x_n)$, and $I$
an ideal in $R.$ Assume that $R/I$ Cohen-Macaulay and ht$(I)=n-3.$
Then $\chi(H^{n-3}_I(R))=s-1,$ where
$H^{n-2}_I(R)=E(R/\mathfrak{m})^s.$
\end{theorem}
\begin{proof}
Let $z$ be a homogeneous linear form which is $R/I$ regular. Note
that \quad  $R/(I+(z))$ will be Cohen-Macaulay. So
$\chi(H^{n-2}_{(I+(z))}(R))=1$ by corollary \ref{H_P^n-1-E-cor}. The
rest of the proof similar to that of Theorem \ref{chi-HP^n-3=s-1}.
\end{proof}

\section{Non-singular surfaces}
The main result of this section is Theorem \ref{non-sing-surf}. We
begin with an easy result.
\begin{lemma}
 Let $R=K[x_1,\cdots,x_n]$
and $\mathfrak{m}=(x_1,\cdots,x_n).$ Let $P\subset R$ be a
homogeneous prime ideal of height $g.$ Suppose Proj$(R/P)$ is smooth
$i.e.$ $(R/P)_Q$ is a regular local ring $\forall$ $Q\not =
\mathfrak{m}$ and $Q$ homogeneous. Then for $i > g$;
$H^i_P(R)=E(R/\mathfrak{m})^{s_i}$ for some $s_i \geq 0$.
\end{lemma}
\begin{proof}
Let $Q$ be a homogeneous prime ideal in $R$ and $P\varsubsetneq Q
\varsubsetneq \mathfrak{m}.$ Then $(R/P)_Q=R_Q/PR_Q$ is a regular
local ring. So $PR_Q=(a_1,\cdots, a_g)$ where $a_1,\ldots,a_g$ is
part of a regular system of parameters of $R_Q$. Therefore
$(H^i_P(R))_Q=H^i_{PR_Q}(R_Q)=0$ for $i > g.$ Thus $H^i_P(R)$ is
supported only at $\mathfrak{m}$ for $i > g.$
\end{proof}
We now extend Theorem \ref{chi-HP^n-3=s-1}.
\begin{theorem}\label{non-sing-surf}
Let  $R=K[x_1,\cdots,x_n]$, $\mathfrak{m}=(x_1,\cdots,x_n)$, and
$P\subset R$ be a homogeneous prime ideal. Suppose $V(P) $ is $r$
dimensional non-singular variety in $\mathbb{P}^{n-1}_K$ with $r\geq
2.$ Then
\begin{align*}
\chi(H^{n-r-1}_P(R))=(-1)^r(-1+\sum_{j\geq n-r}(-1)^{n-j}s_j(P)),
\end{align*}
where $H^j_P(R)=E(R/\mathfrak{m})^{s_j(P)}$ for $j\geq n-r$.
\end{theorem}

\begin{proof}
We prove the result by induction on $r.$ For $r=2$
\begin{align*}
 \chi(H^{n-3}_P(R))=(-1)^2(-1 +
s_{n-2}(P)).
\end{align*}
This is Theorem \ref{chi-HP^n-3=s-1}.

Let $r\geq 3$. We assume the result for non-singular varieties of
$\dim = r-1$ and prove it for non-singular varieties of $\dim =r.$

 Let $H$ be a general
hyperplane. Say $H=V((z)).$ By Bertini's Theorem, $ V(P) \cap H=
V(P+(z))$ is a non-singular variety of dimension $r-1.$

As ht$(P+(z))=n-r,$ we get an exact sequence of the form
\begin{align*}
0\rightarrow H^{n-r-1}_P(R)\rightarrow (H^{n-r-1}_P(R))_z\rightarrow
H^{n-r}_{(P+(z))}(R)\rightarrow H^{n-r}_P(R)\rightarrow 0,
\end{align*}
and $H^j_{P}(R)\cong H^j_{(P+(z))}(R)$ for $j\geq n-r+1=n-(r-1).$
Now by induction hypothesis
\begin{align*}
\chi(H^{j}_{(P+(z))}(R))&=(-1)^{r-1}(-1+\sum_{j\geq
n-r-1}(-1)^{n-j}s_j(P+(z)))\\
&=(-1)^{r-1}(-1+\sum_{j\geq n-r-1}(-1)^{n-j}s_j(P)).
\end{align*}
Note that $\chi((H^{n-r-1}_P(R))_z)=0.$ Therefore
\begin{align*}
\chi(H^{n-r-1}_P(R)) & =(-1)^{r}(-1+\sum_{j\geq
n-r-1}(-1)^{n-j}s_j(P))+s_{n-r}(P)\\
& =(-1)^{r}(-1+\sum_{j\geq n-r}(-1)^{n-j}s_j(P)).
\end{align*}
\end{proof}

\section{examples}
In this section we compute De Rahm homology of two curves in
$K[x,y]$.
\begin{example}
Let $f(x,y)=y+h(x)\in A=K[x,y]$, where $h(x)$ does not have multiple
roots. Then $H_0(\underline{\partial};H_{(f)}^1(A))=0$ and
$H_1(\underline{\partial};H_{(f)}^1(A))\cong K$ and
$H_2(\underline{\partial}; H_{(f)}^1(A))=0.$
\end{example}
\begin{proof}
We have an exact sequence
\begin{center}
$0\rightarrow A\rightarrow A_f\rightarrow H_{(f)}^1(A)\rightarrow
0.$
\end{center}
We know $H_i(\underline{\partial}; A)$ by \cite[Theorem 2.6]{TJ}. So
it suffices to compute $H_i(\underline{\partial}; A_f).$ It is clear
that $H_1(\partial_y; A_f)=K[x].$\\
 First we compute
$H_0(\partial_y; A_f).$ Let $v\in A$, say
 $\deg_yv=m.$ Then $v=\phi_m(x)f^m+\cdots + \phi_1(x)f+\phi_0(x),$
 with $\phi_i(x)\in K[x].$
\begin{align*}
\text{Then} \quad v/f^i =\quad \text{ polynomial} \quad
 +v_2(x)/f + v_1(x)/f^2 + \cdots + v_r(x)/f^r.
 \end{align*}
 Note that for $i\geq 2$, $v_i(x)/f^i = \partial_y(u(x)/f^{i-1})$, where $u(x) = v_i(x)/(i-1).$ So $v/f^i = v_1/f + \partial_y(\theta)$,
 with $\theta \in A_f.$ We prove  $H_0(\partial_y;
 A_f)= K[x]/f.$ For  if $v(x)/f =\partial_y(u/f^i)$ for some $i\geq 1$ and $f$ does not divides $u$.
Then $v(x)/f = \partial_y(u)/f^i + iu/f^{i+1}$. Multiplying by $f^i$
shows that
 $f$ divides $u$ which is a contradiction. So $H_0(\partial_y; A_f)\cong
 K[x]/f.$\\
Next we compute $H_1(\partial_x;H_0(\partial_y; A_f))$. Note that if
$c\in K$ is a constant then
\begin{center}
$\partial_x(c/f)= -h'(x)/f^2 = \partial_y(h'(x)/f)=0$ in
$H_0(\partial_y;A_f).$
\end{center}
 So
$K/f \subseteq H_1(\partial_x;H_0(\partial_y; A_f))$.\\
Let $\partial_x(\phi(x)/f) = 0$ in $H_0(\partial_y; A_f)$. Then
$\partial_x(\phi(x)/f)=\partial_y(u(x,y)/f^i)$ for some $u(x,y)\in
 A$ and $f$ does not divides $u.$ By computing both sides we get that $f$ divides  $u$ if $i>1,$ which is a
 contradiction.  So $i=1$ and
\begin{center}
  $\phi(x) '/f + \phi(x).h'(x)/f^2 = \partial_y(u)/f - u/f^2.$
\end{center}
  Multiplying by
 $f^2$ gives,
\begin{center}
  $f.\phi(x)'+\phi(x)h'(x)=f.\partial_y(u)-u.$
\end{center}
  This
 implies $f$ divides $u-\phi(x).h'(x)$. Say $u-\phi(x).h'(x)=f.g$, for some $g=g(x,y)\in
 A$. Then $\partial_y(u) = \partial_y( f )g + f\partial_y( g) = g + f\partial_y( g).$
 So that
\begin{align*}
 f.\phi(x)'=(g+f\partial_y( g))f-fg.
\end{align*}
  So $\phi(x)'=g + f\partial_y( g) - g$,   $i.e. f\partial_y( g) = \phi(x)'$.
  This shows that
 $f$ divides $\phi(x)'$, this is possible only when $\phi(x)'=0$. So
 $\phi(x)=c$ constant.
 Thus
\begin{align*}
  H_1(\partial_x;K[x]/f)=K/f\cong K.
 \end{align*}

Now we compute $H_0(\partial_x; K[x]/f )$. Let $\phi(x)/f\in
H_0(\partial_y; A_f )=K[x]/f$. One can easily verify that
\begin{align*}
\phi(x)/f=\partial_x(\psi(x)/f)+\partial_y(u/f)
\end{align*}
where $\psi(x)=\int \phi(x)dx$ and $u= -\psi(x)h'(x).$ So
$H_0(\partial_x; K[x]/f) = 0$.

Since we have an exact sequence of the form
\begin{center}
$0\rightarrow H_0(\partial_x ;H_i(\partial_y;A_f))\rightarrow
H_i(\partial_x,\partial_y; A_f)\rightarrow
H_1(\partial_x;H_{i-1}(\partial_y; A_f))\rightarrow 0$, for all $i.$
\end{center}
As $H_0(\partial_x;H_1(\partial_y;A_f))=H_0(\partial_x; K[x])=0.$
Therefore
\begin{align*}
H_0(\partial_x,\partial_y; A_f) & \cong H_0(\partial_x; H_0(\partial_y ; A_f))=0,  and\\
H_1(\partial_x,\partial_y; A_f)& \cong H_1(\partial_x;
H_0(\partial_y; A_f))\cong K.\\
\text{Also \quad note} \quad H_2(\partial_x,\partial_y;
A_f)&=H_1(\partial_x; H_1(\partial_y; A_f))\cong K.
\end{align*}
 As $0\rightarrow A\rightarrow A_f\rightarrow H^1 _{(f)}(A)\rightarrow
 0$ is exact, by taking Koszul homology with respect to $\partial_x,\partial_y $ and
  using Lemma 2.7 of \cite{TJ} we get $H_2(\underline{\partial};
  H^1_{(f)}(A))=0$ and   $H_i(\underline{\partial}; A_f)\cong H_i(\underline{\partial}; H^1_{(f)}(A))$
for $i=0,1.$ Hence the result.
\end{proof}
Our next example is:
\begin{example}\label{application}
Let $f(x,y)=xy+1\in A=K[x,y]$, then as $K-$ vector spaces
$H_0(\underline{\partial}; H_{(f)}^1(R))\cong K$,
$H_1(\underline{\partial} ; H_{(f)}^1(R))\cong K$ and
$H_2(\underline{\partial} ; H_{(f)}^1(R))=0$.
\end{example}
\begin{proof}
As in the above example it is enough to compute
$H_i(\underline{\partial}; A_f).$ First note that $H_1(\partial_y;
A_f)=K[x]$, $f-1=xy,$
$\partial_x(f)=y$ and $\partial_y(f)=x.$\\
First we compute $H_0(\partial_y; A_f).$ Let $a=a(x,y)\in A$. Then
\begin{center}
$\partial_y(a/f^i)=\partial_y(a)/f^i - ia.x/f^{i+1}$ \quad  for all
$i\geq 1.$
\end{center}
\textit{Case}$(1):$ Let $a(x, y)= \phi(x).$ Then
$\partial_y(\phi(x)/f^i)=-i\phi(x) x/f^{i+1}$ for all $i\geq 1.$
This shows that
\begin{equation}\label{x-van}
x \phi(x)/f^i \equiv 0 \quad  \text{in} \quad H_0(\partial_y; A_f)
\quad \text{for all} \quad i\geq 2.
\end{equation}
\begin{align*}
\text{ We have}\quad \partial_y(y.\phi(x)/f^i)&=\phi(x)/f^i - i\phi(x).xy/f^{i+1}\\
&=\phi(x)/f^i - i\phi(x)(f-1)/f^{i+1}\\
&=\phi(x)/f^i - i\phi(x)/f^{i} + i\phi(x)/f^{i+1}.
\end{align*}
when $i=1$ $\phi(x)/f^2 \equiv 0$ in $H_0(\partial_y; A_f).$ Thus
$i\phi(x)/f^{i+1}\equiv (i-1)\phi(x)/f^i$ in $H_0(\partial_y; A_f).$
Continuing in this way we get
\begin{equation}\label{x-equ}
\phi(x)/f^{i}\equiv k_0\phi(x)/f \quad \text{in} \quad
H_o(\partial_y; A_f), \text {for all} \quad i\geq 2 \quad k_0\in K.
\end{equation}

Let $a(x, y)=y^m, m\geq 1.$ Then
 \begin{align*}
 \partial_y(y^m/f^i) & =my^{m-1}/f^i -ixy^m/f^{i+1}\\
& = my^{m-1}/f^i-i(f-1)y^{m-1}/f^{i+1}\\
 & =(m-i)y^{m-1}/f^i + iy^{m-1}/f^{i+1}.
\end{align*}
This shows that $y^m/f^i \equiv k_1y^m/f^{i-1}$ in $ H_o(\partial_y;
A_f)$ for some $k_1\in K$,for all $i\geq 2.$ Continuing in this way
we get
\begin{equation}\label{y-equ}
y^m/f^i \equiv ky^m/f \quad \text{in} \quad H_0(\partial_y; A_f)
\quad \forall \quad i\geq 1, \quad \text{for some } \quad k\in K.
\end{equation}

\textit{Case}$(2):$ Write $a=a(x, y) = \phi(x) + \psi(y) +
xya_1(x,y).$
\begin{align*}
\text{So} \quad a/f^i & = \phi_1(x)/f^i +
\psi_1(y)/f^i + a_1(f-1)/f^i\\
& \equiv \phi_1(x)/f + \psi_1(y)/f + a_1/f^{i-1} - a_1/f^i \quad
\text{in} \ H_0(\partial_y; A_f)\ \text{(by \ref{x-equ} and
\ref{y-equ})}
\end{align*}
Write $a_1= \phi_2(x)+ \psi_2(y) + xya_2(x,y)$ and do the same as
above to get
\begin{center}
$a/f^i \equiv \phi(x)/f + \psi(y)/f \quad $  in $\quad
H_0(\partial_y; A_f).$
\end{center}
Therefore $H_0(\partial_y; A_f) = < \bar{\phi(x)/f}, \bar{
\psi(y)/f} >$ as a $K$-vector space.

Now we compute $H_0(\partial_x; H_0(\partial_y; A_f)).$  Let $x^n/f
\in H_0(\partial_y; A_f)$ for some $n\geq 1.$
\begin{align*}
\text{Since} \quad \partial_x(x^{n+1}/f)& = (n+1)x^n/f -
x^{n+1}y/f^2\\
& = (n+1)x^n/f - x^n(f-1)/f^2\\
& = nx^n/f + x^n/f^2.
\end{align*}
By (\ref{x-van} ) $x^n/f^2\equiv 0$ in $H_0(\partial_y; A_f).$ So
$x^n/f \equiv
0$ in $H_0(\partial_x; H_0(\partial_y; A_f))$ for all $n\geq 1.$\\
Let $y^m/f \in H_0(\partial_y; A_f).$
\begin{align*}
\text{Since} \quad \partial_y(y^{m+1}/f)& =(m+1)y^m/f - xy^{m+1}/f^2\\
& =(m+1)y^m/f - y^m(f-1)/f^2\\
& =my^m/f + y^m/f^2.
\end{align*}
Note that $\partial_x(-y^{m-1}/f)=y^m/f^2.$ So
\begin{align*}
 y^m/f &=1/m(\partial_y(y^{m+1}/f)+ \partial_x(-y^{m-1}/f))\\
& \equiv 0 \quad \text{in}\quad H_0(\partial_x; H_0(\partial_y;
A_f))\quad \text{for all} \quad  m\geq 1.
\end{align*}
 Therefore $H_0(\partial_x;
H_0(\partial_y; A_f))=<1/f>$ as a $K$ vector space. In particular
\begin{equation}\label{H_0-dim}
 \dim_KH_0(\partial_x; H_0(\partial_y; A_f))\leq
1.
\end{equation}
 Since we have an exact sequence of the form
\begin{center}
$0\rightarrow H_0(\partial_x; H_i(\partial_y;A_f))\rightarrow
H_i(\underline{\partial}; A_f)\rightarrow H_1(\partial_x;
H_{i-1}(\partial_y; A_f))\rightarrow 0$, for all $i.$
\end{center}

As $H_0(\partial_x; H_1(\partial_y;A_f))=H_0(\partial_x; K[x])=0.$
Therefore
\begin{equation}\label{H_0-equ}
H_0(\underline{\partial}; A_f) \cong H_0(\partial_x; H_0(\partial_y;
A_f)),
\end{equation}
\begin{equation}\label{H_1-equ}
H_1(\underline{\partial}; A_f) \cong H_1(\partial_x; H_0(\partial_y;
A_f)).\quad \text{Note}\quad H_0(\underline{\partial};A_f)\cong
H_0(\underline{\partial}; H_{(f)}^1(A)).
\end{equation}

Note that the points at $\infty$ of $f$ are $[1:0:0]$ and $[0:1:0].$
Therefore  by Theorem \ref{H_0-pts at inf} we get
\begin{center}
 $2\leq
1+\dim_KH_0(\underline{\partial}; H_{(f)}^1(A)),$ and by
(\ref{H_0-dim}) we get  $\dim_KH_0(\underline{\partial}; A_f)=1.$
\end{center}
Now we compute $H_1(\partial_x; H_0(\partial_y; A_f)).$

Let $\xi = \phi(x)/f + \psi(y)/f \in H_1(\partial_x ; H_0(\partial_y
; A_f)).$ Then
\begin{align*}
\partial_x(\phi(x)/f + \psi(y)/f)&=0 \quad \text{in} \quad H_0(\partial_y ;
A_f).\\
\Rightarrow \quad \partial_x(\phi(x)/f)+ \partial_x(\psi(y)/f) & =
\partial_y(a(x,y)/f^i), \quad i\geq 1 \quad f\nmid a(x,y).
\end{align*}
By computing both sides we get $i=1.$ Thus
\begin{align*}
\partial_x(\phi(x)/f)+ \partial_x(\psi(y)/f) & =
\partial_y(a(x,y)/f).\\
 \Rightarrow \quad \partial_x(\phi(x))/f
- \phi(x)y/f^2 - \psi(y)y/f^2 & = \partial_y(a)/f - ax/f^2.\\
\Rightarrow \quad f\partial_x(\phi(x))-\phi(x)y - \psi(y)y & =
f\partial_y(a)-ax.\\
\Rightarrow \quad f(\partial_x(\phi(x))-\partial_y(a)) &
=\phi(x)y+\psi(y)y - ax.
\end{align*}
Applying $\partial_y$ we get
\begin{align*}
x(\partial_x(\phi(x))-\partial_y(a)) - f\partial_y(\partial_y(a)) &
= \phi(x) + \psi(y) + y\partial_y(\psi(y)) - x\partial_y(a).\\
\Rightarrow \quad x\partial_x(\phi(x)) - f\partial_y(\partial_y(a))
& = \phi(x)+ \psi(y) + y\partial_y(\psi(y)).
\end{align*}
Thus $f$ divides $(x\partial_x(\phi(x))-\phi(x))-( \psi(y) +
y\partial_y(\psi(y))).$ As $f=xy+1.$ We get
$(x\partial_x(\phi(x))-\phi(x))-( \psi(y) +
y\partial_y(\psi(y)))=0.$ Thus $x\partial_x(\phi(x))-\phi(x)=0$ and
$\psi(y) + y\partial_y(\psi(y))=0.$ Hence $\phi(x)= cx$ where $c \in
K$ and $\psi(y)=0.$
 Therefore
\begin{center}
$H_1(\partial_x; H_0(\partial_y; A_f))= < x/f
>.$
\end{center}
 Also $x/f\neq 0$ in
$H_0(\partial_y; A_f).$ For if $x/f =
\partial_y(u(x,y)/f^i)$ such that $f$ does not divide $u.$ Then
$xf^{i}=\partial_y(u) f-ixu.$ Thus $f$ divides $-ixu.$ This is a
contradiction as $f$ does not divides $u$ and $x.$ Therefore
$H_1(\partial_x; H_0(\partial_y; A_f))\neq 0.$ So $\dim_K
H_1(\partial_x; H_0(\partial_y; A_f))=1.$ Therefore
by(\ref{H_1-equ}) $\dim_KH_1(\underline{\partial}; A_f)=1.$ Hence
the result.
\end{proof}

\end{document}